\title{Biased Games On Random Boards}
\author{\quad{Asaf Ferber}\thanks{School of Mathematical Sciences,
Raymond and Beverly Sackler Faculty of Exact Sciences, Tel Aviv
University, Tel Aviv, 69978, Israel. Email:
ferberas@post.tau.ac.il.} \quad{Roman Glebov}\thanks{Institut
f\"{u}r Mathematik, Freie Universit\"at Berlin, Arnimallee 3-5,
D-14195 Berlin, Germany. Email: glebov@math.fu-berlin.de. Research
supported by DFG within the research training group ``Methods for
Discrete Structures".} \quad{Michael Krivelevich}\thanks{School of
Mathematical Sciences, Raymond and Beverly Sackler Faculty of Exact
Sciences, Tel Aviv University, Tel Aviv, 69978, Israel. Email:
krivelev@post.tau.ac.il. Research supported in part by USA-Israel
BSF Grant 2010115 and by grants 1063/08, 912/12 from Israel Science
Foundation.} \quad{Alon Naor}
\thanks{School of Mathematical Sciences, Raymond and Beverly Sackler
Faculty of Exact Sciences, Tel Aviv University, Tel Aviv, 69978,
Israel. Email: alonnaor@post.tau.ac.il.}}

\documentclass[11pt]{article}
\usepackage{amsmath,amssymb,latexsym,color,epsfig,a4, enumerate}
\usepackage{bbm, mathtools}

\newif\ifnotesw\noteswtrue

\newtheorem{theorem}{Theorem}[section]
\newtheorem{lemma}[theorem]{Lemma}
\newtheorem{corollary}[theorem]{Corollary}

\newtheorem{claim}[theorem]{Claim}

\newtheorem{conjecture}[theorem]{Conjecture}

\newtheorem{definition}[theorem]{Definition}

\let\eps=\varepsilon
\let\theta=\vartheta
\let\rho=\varrho
\let\sigma=\varsigma

\let\polishlcross=\l
\def\l{\ifmmode\ell\else\polishlcross\fi}

\newcommand{\cF}{{\cal F}}

\newcommand{\gnp}{\ensuremath{G(n,p)}}
\newcommand{\Bin}{\ensuremath{\textrm{Bin}}}

\newcommand{\danger}{\mathsf {dang}}
\newcommand{\avdanan}{\overline{\danger}}

\newcommand{\Exp}{\mathbb{E}}

\newenvironment{proof}{\noindent{\bf Proof.\,}}{\hfill$\Box$}
\allowdisplaybreaks[4]

\begin{document}
\maketitle

\begin{abstract}
In this paper we analyze biased Maker-Breaker games and
Avoider-Enforcer games, both played on the edge set of a random
board $G\sim \gnp$. In Maker-Breaker games there are two players,
denoted by Maker and Breaker. In each round, Maker claims one
previously unclaimed edge of $G$ and Breaker responds by claiming
$b$ previously unclaimed edges. We consider the Hamiltonicity game,
the perfect matching game and the $k$-vertex-connectivity game,
where Maker's goal is to build a graph which possesses the relevant
property. Avoider-Enforcer games are the reverse analogue of
Maker-Breaker games with a slight modification, where the two
players claim at least $1$ and at least $b$ previously unclaimed
edges per move, respectively, and Avoider aims to avoid building a
graph which possesses the relevant property.

Maker-Breaker games are known to be ``bias-monotone", that is, if
Maker wins the $(1,b)$ game, he also wins the $(1,b-1)$ game.
Therefore, it makes sense to define the \emph{critical bias} of a
game, $b^*$, to be the ``breaking point" of the game. That is, Maker
wins the $(1,b)$ game whenever $b\leq b^*$ and loses otherwise. An
analogous definition of the critical bias exists for
Avoider-Enforcer games: here, the critical bias of a game $b^*$ is
such that Avoider wins the $(1,b)$ game for every $b > b^*$, and
loses otherwise.

We prove that, for every $p=\omega(\frac{\ln n}{n})$, $G\sim\gnp$ is
typically such that the critical bias for all the aforementioned
Maker-Breaker games is asymptotically $b^*=\frac{np}{\ln n}$. We
also prove that in the case $p=\Theta(\frac{\ln n}{n})$, the
critical bias is $b^*=\Theta(\frac{np}{\ln n})$. These results
settle a conjecture of Stojakovi\'c and Szab\'o. For
Avoider-Enforcer games, we prove that for $p=\Omega(\frac{\ln
n}{n})$, the critical bias for all the aforementioned games is
$b^*=\Theta(\frac{np}{\ln n})$.
\end{abstract}

\section{Introduction}
Let $X$ be a finite set and let $\mathcal F\subseteq 2^X$. In an
$(a,b)$ Maker-Breaker game $\mathcal F$, the two players -- Maker
and Breaker -- alternately claim $a$ and $b$ previously unclaimed
elements of \emph{the board} $X$, respectively. Maker's goal is to
claim all the elements of some \emph{target set} $F\in \mathcal F$.
If Maker does not fully claim any target set by the time all board
elements are claimed, then Breaker wins the game. When $a=b=1$, the
game is called \emph{unbiased}, otherwise it is called
\emph{biased}. It is easy to see that being the first player is
never a disadvantage in a Maker-Breaker game: indeed, suppose the
first player has some strategy as the second player. He can play
arbitrarily in his first move and pretend that he didn't make this
move and he now starts a new game as a second player; whenever his
strategy tells him to claim some edge he had previously claimed he
just claims arbitrarily some edge. So, in order to prove that Maker
wins a certain game, it is enough to prove that he can win as a
second player. Throughout the paper we assume that Maker is the
\emph{second} player to move.

In an $(a,b)$ Avoider-Enforcer game played on a hypergraph $\mathcal
F\subseteq 2^X$, the two players are called Avoider and Enforcer,
alternately claim \emph{at least} $a$ and \emph{at least} $b$
previously unclaimed elements of the board $X$ per move,
respectively. Avoider loses the game if at some point during the
game he fully claims all the elements of some target set $F\in
\mathcal F$. Otherwise, Avoider wins.

In both Maker-Breaker games and Avoider-Enforcer games, we may
assume that there are no $F_1, F_2 \in \cF$ such that $F_1 \subset
F_2$, since in this case Maker wins (or Avoider loses) once he
claims all the elements in $F_1$, and so the two $(a,b)$ games $\cF$
and $\cF \setminus \{F_2\}$ are identical.

It is natural to play positional games on the edge set of a graph
$G$. In this case, the board is $X=E(G)$, and the target sets are
all the edge sets of subgraphs $H\subseteq G$ which possess some
given monotone increasing graph property $\mathcal P$. In the
\emph{connectivity} game ${\mathcal C}(G)$, the target sets are all
edge sets of spanning trees of $G$. In the \emph{perfect matching}
game ${\mathcal M}(G)$ the target sets are all sets of $\lfloor
|V(G)|/2 \rfloor$ independent edges of $G$. Note that if $n$ is odd,
then such a matching covers all vertices of $G$ but one. In the
\emph{Hamiltonicity} game ${\mathcal H}(G)$ the target sets are all
edge sets of Hamilton cycles of $G$. Given a positive integer $k$,
in the \emph{$k$-connectivity} game ${\mathcal C}^k(G)$ the target
sets are all edge sets of $k$-vertex-connected spanning subgraphs of
$G$.

Maker-Breaker games played on the edge set of the complete graph
$K_n$ are well studied. In this case, many natural unbiased games
are drastically in favor of Maker (see, e.g., \cite{Lehman},
\cite{HKSS}, \cite{HS}, \cite{FH}). Hence, in order to even out the
odds, it is natural to give Breaker more power by increasing his
\emph{bias} (that is, to play a $(1,b)$ game instead of a $(1,1)$
game), and/or to play on different types of boards.

Maker-Breaker games are \emph{bias monotone}. That means that if
Maker wins some game with bias $(a,b)$, he also wins this game with
bias $(a',b')$, for every $a' \geq a$, $b' \leq b$. Similarly, if
Breaker wins a game with bias $(a,b)$, he also wins this game with
bias $(a',b')$, for every $a' \leq a$, $b' \geq b$. Avoider-Enforcer
games are also bias monotone in the version considered in this paper
(this version is called the \emph{monotone} version, as opposed to
the \emph{strict} version, where Avoider and Enforcer claim exactly
$a$ and $b$ elements per move, respectively. The strict version is
not bias monotone.). It means that if Avoider wins some game with
bias $(a,b)$, he also wins this game with bias $(a',b')$, for every
$a' \leq a$, $b' \geq b$, and that if Enforcer wins a game with bias
$(a,b)$, he also wins this game with bias $(a',b')$, for every $a'
\geq a$, $b' \leq b$.

This bias monotonicity allows us to define the \emph{critical bias}
(also referred to as the \emph{threshold bias}): for a given game
$\cF$, the critical bias $b^*$ is the value for which Maker wins the
game $\cF$ with bias $(1,b)$ for every $b < b^*$, and Breaker wins
the game $\cF$ with bias $(1,b)$ for every $b \geq b^*$. Similarly,
this is the value for which Avoider wins the game $\cF$ with bias
$(1,b)$ for every $b \geq b^*$, and that Enforcer wins the game
$\cF$ with bias $(1,b)$ for every $b < b^*$.

In their seminal paper \cite{CE}, Chvat\'{a}l and Erd\H{o}s proved
that playing the $(1,b)$ connectivity game on the edge set of the
complete graph $K_n$, for every $\eps > 0$, Breaker wins for every
$b\geq \frac{(1+\varepsilon)n}{\ln n}$, and Maker wins for every
$b\leq \frac{n}{(4+\varepsilon)\ln n}$. They conjectured that
$b=\frac{n}{\ln n}$ is (asymptotically) the threshold bias for this
game. Gebauer and Szab\'o proved in \cite{GS} that this is indeed
the case. Later on, Krivelevich proved in \cite{K} that
$b=\frac{n}{\ln n}$ is also the threshold bias for the Hamiltonicity
game.

Stojakovi\'c and Szab\'o suggested in \cite{SS} to play
Maker-Breaker games on the edge set of a random board $G\sim \gnp$.
In this well known and well studied model, the graph $G$ consists of
$n$ labeled vertices, and each pair of vertices is chosen to be an
edge in the graph independently with probability $p$. They examined
some games on this board such as the connectivity game, the perfect
matching game, the Hamiltonicity game and building a $k$-clique
game. Since then, much progress has been made in understanding
Maker-Breaker games played on $G\sim \gnp$. For example, it was
proved in \cite{BFHK} that for $p=\frac{(1+o(1))\ln n}{n}$, $G\sim
\gnp$ is typically (i.e. with probability tending to $1$ as $n$
tends to infinity) such that Maker wins the $(1,1)$ games $\mathcal
M(G)$, $\mathcal H(G)$ and $\mathcal C_k(G)$. Moreover, the proofs
in \cite{BFHK} are of a ``hitting time" type. It means that in the
random graph process (see \cite{Boll}), typically at the moment the
graph reaches the needed minimum degree for Maker to win the desired
game, Maker indeed wins this game. Later on, in \cite{CFKL}, fast
winning strategies for Maker in various games played on $G\sim \gnp$
were considered, and in \cite{MS} a hitting time result was
established for the ``building a triangle" game, and it was proved
that the threshold probability for the property ``Maker can build a
$k$-clique" game is $p=\Theta\left(n^{-2/(k+1)}\right)$.

In \cite{SS}, Stojakovi\'c and Szab\'o conjectured the following:

\begin{conjecture}[\textbf{\cite{SS}, Conjecture 1}]\label{conjecture1}
There exists a constant $C$ such that for every $p\geq \frac{C\ln
n}{n}$, a random graph $G\sim \gnp$ is typically such that the
threshold bias for the game $\mathcal H(G)$ is
$b^*=\Theta\left(\frac{np}{\ln n}\right)$.
\end{conjecture}

In this paper we prove Conjecture~\ref{conjecture1}, and in fact,
for $p=\omega\left(\frac{\ln n}{n}\right)$ we prove the following
stronger statement:

\begin{theorem} \label{main1}
Let $p=\omega\left(\frac{\ln n}{n}\right)$. Then $G\sim \gnp$ is
typically such that $\frac{np}{\ln n}$ is the asymptotic threshold
bias for the games $\mathcal M(G)$, $\mathcal H(G)$ and $\mathcal
C_k(G)$.
\end{theorem}

In order to prove Theorem~\ref{main1} we prove the following two
theorems:

\begin{theorem}\label{BreakerIsolatingAvertex}
Let $0 \leq p \leq 1$, $\varepsilon>0$ and $b \geq
(1+\eps)\frac{np}{\ln n}$. Then $G\sim \gnp$ is typically such that
in the $(1,b)$ Maker-Breaker game played on $E(G)$, Breaker has a
strategy to isolate a vertex in Maker's graph, as a first or a
second player.
\end{theorem}

\begin{theorem} \label{mainDense}
Let $p=\omega\left(\frac{\ln n}{n}\right)$, $\varepsilon>0$ and
$b=(1-\varepsilon)\frac{np}{\ln n}$. Then  $G\sim \gnp$ is typically
such that Maker has a winning strategy in the $(1,b)$ games
$\mathcal M(G)$, $\mathcal H(G)$, and $\mathcal C_k(G)$ for a fixed
positive integer $k$, as a first or a second player.
\end{theorem}

In the case $p=\Theta\left(\frac{\ln n}{n}\right)$ we establish two
non-trivial bounds for the critical bias $b^*$. This also settles
Conjecture~\ref{conjecture1} for this case but does not determine
the exact value of $b^*$ (notice that in this case, $b^*$ is a
constant!).

\begin{theorem} \label{mainSparse}
Let $p=\frac{c\ln n}{n}$, where $c>1600$ and let $\varepsilon>0$.
Then $G\sim \gnp$ is typically such that the threshold bias for the
games $\mathcal M(G)$, $\mathcal H(G)$ and $\mathcal C_k(G)$ lies
between $c/10$ and $c+\varepsilon$.
\end{theorem}

\textbf{Remark:} In the terms of Theorem~\ref{mainSparse}, if $1 < c
\leq 1600$, we get by Theorem~\ref{BreakerIsolatingAvertex} that
$b^* \leq c+\eps$, and by the main result of \cite{BFHK} that $b^* >
1$, so indeed $b^* = \Theta(\frac{np}{\ln n})$ in this case as well.

We also consider the analogous Avoider-Enforcer games played on the
edge set of a random board $G \sim \gnp$. Here Avoider aims to avoid
claiming all the edges of a graph which contains a perfect matching,
a Hamilton cycle, or that is $k$-connected, (according to the game),
and Enforcer tries to force him claiming all the edges of such a
subgraph. We prove the following analog of
Conjecture~\ref{conjecture1}:

\begin{theorem}\label{main2}
Let $\frac{70000\ln n}{n} \leq p \leq 1$. A random graph $G \sim
\gnp$ is typically such that the asymptotic threshold bias for the
(1,b) Avoider-Enforcer games $\mathcal M(G)$, $\mathcal H(G)$ and
$\mathcal C_k(G)$ (for a fixed positive integer $k$) is
$b^*=\Theta(\frac{np}{\ln n})$.
\end{theorem}

As in the Maker-Breaker case, we divide our result into two separate
theorems, one which establishes Avoider's win, and one which
establishes Enforcer's win:

\begin{theorem}\label{AvoiderWin}
Let $0 \leq p \leq 1$ and $b \geq \frac{25np}{\ln n}$. Then $G\sim
\gnp$ is typically such that in the $(1,b)$ Avoider-Enforcer game
played on $E(G)$, Avoider has a strategy to isolate a vertex in his
graph, as a first or a second player.
\end{theorem}

\begin{theorem} \label{EnforcerWin}
Let $\frac{70000\ln n}{n} \leq p \leq 1$ and $b \leq
\frac{np}{20000\ln n}$. Then  $G\sim \gnp$ is typically such that
Enforcer has a winning strategy in the $(1,b)$ games $\mathcal
M(G)$, $\mathcal H(G)$ and $\mathcal C_k(G)$ (for every positive
integer $k$), as a first or a second player.
\end{theorem}

\subsection{Notation and terminology}
Our graph-theoretic notation is standard and follows that of
\cite{West}. In particular, we use the following:

For a graph $G$, let $V=V(G)$ and $E=E(G)$ denote its sets of
vertices and edges, respectively. For subsets $U,W \subseteq V$, and
for a vertex $v \in V$, we denote by $E(U)$ all the edges with both
endpoints in $U$, by $E(U,W)$ all the edges with one endpoint in $U$
and one endpoint in $W$, and by $E(v,U)$ all the edges with one
endpoint being $v$ and one endpoint in $U$. We further denote
$e(U):=|E(U)|$, $e(U,W):=|E(U,W)|$ and $e(v,U):=|E(v,U)|$.

For a subset $U\subseteq V$ we denote by $N(U)$ the \emph{external}
neighborhood of $U$, that is: $N(U):=\{v \in V \setminus U: \exists
u \in U \text{ s.t. } uv \in E\}$.

Assume that some Maker-Breaker game, played on the edge set of some
graph $G$, is in progress. At any given moment during the game, we
denote the graph formed by Maker's edges by $M$, the graph formed by
Breaker's edges by $B$, and the edges of $G\setminus (M\cup B)$ by
$F$. For any vertex $v \in V$, $d_M(v)$ and $d_B(v)$ denote the
degree of $v$ in $M$ and in $B$, respectively. The edges of $G
\setminus (M\cup B)$ are called \emph{free edges}, and $d_F(v)$
denotes the number of free edges incident to $v$, for any $v \in V$.

Whenever we say that $G \sim \gnp$ \emph{typically} has some
property, we mean that $G$ has that property with probability
tending to $1$ as $n$ tends to infinity.

We use the following notation throughout this
paper:$$f(n):=\frac{np}{\ln n}.$$

For the sake of simplicity and clarity of presentation, and in order
to shorten some of our proofs, no real effort has been made here to
optimize the constants appearing in our results. We also omit floor
and ceiling signs whenever these are not crucial. Most of our
results are asymptotic in nature and whenever necessary we assume
that $n$ is sufficiently large.

%
%
%
\section{Auxiliary results}

In this section we present some auxiliary results that will be used
throughout the paper.

\subsection{Binomial distribution bounds}
We use extensively the following well known bound on the lower and
the upper tails of the Binomial distribution due to Chernoff (see,
e.g., \cite{AS}):

\begin{lemma}\label{Che}
If $X \sim \Bin(n,p)$, then
\begin{itemize}
    \item $\Pr\left(X<(1-a)np\right)<\exp\left(-\frac{a^2np}{2}\right)$ for every $a>0.$
    \item $\Pr\left(X>(1+a)np\right)<\exp\left(-\frac{a^2np}{3}\right)$ for every $0 < a < 1.$
\end{itemize}
\end{lemma}

The following is a trivial yet useful bound:
\begin{lemma}\label{l:Che}
Let $X \sim \Bin(n,p)$ and $k \in \mathbb{N}$. Then $$\Pr(X \geq k)
\leq \left(\frac{enp}{k}\right)^k.$$
\end{lemma}
\begin{proof}
$\Pr(X \geq k) \leq \binom{n}{k}p^k \leq
\left(\frac{enp}{k}\right)^k$.
\end{proof}

\subsection{Basic positional games results}
\subsubsection{Maker-Breaker games}
The following fundamental theorem, due to Beck~\cite{BeckBook}, is a
useful sufficient condition for Breaker's win in the $(a,b)$ game
$(X, {\mathcal F})$. It will be used in the proof of Theorem
\ref{mainDense}.

\begin{theorem}[\textbf{\cite{BeckBook}, Theorem 20.1}]\label{bwin}
Let $X$ be a finite set and let ${\mathcal F} \subseteq 2^X$.
Breaker, as a first or a second player, has a winning strategy in
the $(a,b)$ game $(X, {\mathcal F})$, provided that:
$$\sum_{F \in {\mathcal F}}(1+b)^{-|F| / a} < \frac{1}{1+b}.$$
\end{theorem}

While Theorem~\ref{bwin} simply shows that Breaker can win certain
games, the following lemma shows that Maker can win certain games
quickly (see \cite{BeckBook}):

\begin{lemma} [Trick of fake moves] \label{fakemoves}
Let $X$ be a finite set and let ${\mathcal F} \subseteq 2^X$. Let
$b' < b$ be positive integers. If Maker has a winning strategy for
the $(1,b)$ game $(X, {\mathcal F})$, then he has a strategy to win
the $(1,b')$ game $(X, {\mathcal F})$ within
$\lceil{\frac{|X|}{b+1}}\rceil$ moves.
\end{lemma}

The main idea of the proof of Lemma~\ref{fakemoves} is that, in
every move of the $(1,b')$ game $(X, {\mathcal F})$, Maker (in his
mind) gives Breaker $b - b'$ additional board elements. The
straightforward details can be found in~\cite{BeckBook}.

Recall the classic \emph{box game} which was first introduced by
Chv\'atal and Erd\H{o}s in~\cite{CE}. In the Box Game \emph{$Box(m,
\ell, b)$} there are $m$ pairwise disjoint boxes $A_1, \ldots, A_m$,
each of size $\ell$. In every round, the first player, called
\emph{BoxMaker}, claims $b$ elements of $\bigcup_{i=1}^m A_i$ and
then the second player, called \emph{BoxBreaker}, destroys one box.
BoxMaker wins the game $Box(m, \ell, b)$ if and only if he is able
to claim all elements of some box before it is destroyed. We use the
following theorem which was proved in \cite{CE}:

\begin{theorem} \label{boxgame}
Let $m,\ell$ be two integers. Then, BoxMaker wins the game $Box(m,
\ell, b)$ for every $b>\frac{\ell}{\ln m}$.
\end{theorem}

\subsubsection{Avoider-Enforcer games}
Similarly to Theorem~\ref{bwin}, we have the following sufficient
condition for Avoider's win, which was proved in \cite{HKS}:

\begin{lemma} [\cite{HKS}, Theorem 1.1] \label{AvoidersCrit}
Let $X$ be a finite set and let ${\mathcal F} \subseteq 2^X$.
Avoider, as a first or a second player, has a winning strategy in
the $(a,b)$ game $(X, {\mathcal F})$, provided that:
$$ \sum_{F\in {\mathcal F}}
\left(1+\frac{1}{a}\right)^{-|F|}<\left(1+\frac{1}{a}\right)^{-a}.$$
\end{lemma}

In the proof of Theorem~\ref{AvoiderWin} we use the Avoider-Enforcer
version of the box game -- \emph{monotone}-$rBox(b_1, \ldots,
b_n,(p,q))$ which was analyzed in \cite{FKN}. In this game there are
$n$ disjoint boxes of sizes $1 \leq b_1 \leq \ldots \leq b_n$,
Avoider claims at least $p$ elements per move, Enforcer claims at
least $q$ elements per move, and Avoider loses if and only if he
claims all the elements in some box by the end of the game. The
following lemma can be easily derived from Theorem 1.7 and Remark
3.2 in \cite{FKN}:

\begin{lemma}\label{EnforcerWinBox}
Let $b,k$ be positive integers. For every integer $n \geq 2e^{k/b}$
and for every sequence of integers $1 \leq b_1 \leq \ldots \leq b_n
\leq k$, Enforcer wins the game
monotone-$rBox(b_1,\ldots,b_n,(b,1))$ as a first or a second player.
\end{lemma}

%

%
%
%

\subsection{$(R,c)$-Expanders}

\begin{definition}
For every $c>0$ and every positive integer $R$ we say that a graph
$G=(V,E)$ is  an \emph{$(R,c)$-expander} if $|N(U)|\geq c|U|$ for
every subset of vertices $U\subseteq V$ such that $|U|\leq R$.
\end{definition}

In the proof of Theorem~\ref{mainDense} Maker builds an expander and
then he turns it into a Hamiltonian graph. In order to describe the
relevant connection between Hamiltonicity and $(R,c)$-expanders, we
need the notion of \emph{boosters}.

Given a graph $G$, we denote by $\ell(G)$ the maximum length of a
path in $G$.

\begin{definition}\label{d:booster}
For every graph $G$, we say that a non-edge $uv\notin E(G)$ is a
\emph{booster} with respect to $G$, if either $G \cup \{uv\}$ is
Hamiltonian or $\ell(G \cup \{uv\})>\ell(G)$. We denote by
${\mathcal B}_G$ the set of boosters with respect to $G$.
\end{definition}
The following is a well-known property of $(R,2)$-expanders (see
e.g. \cite{FK}).
\begin{lemma}\label{l:expander_booster}
If $G$ is a connected non-Hamiltonian $(R,2)$-expander, then
$|{\mathcal B}_G|\geq R^2/2$.
\end{lemma}

Our goal is to show that during a game on an appropriate graph $G$,
assuming Maker can build a subgraph of $G$ which is an
$(R,2)$-expander, he can also claim sufficiently many such boosters,
so that his $(R,2)$-expander becomes Hamiltonian. In order to do so,
we need the following lemma:

\begin{lemma} \label{l:enoughboosters}
Let $a>0$ and $p > \frac{800a\ln n}{n}$. Then $G\sim \gnp$ is
typically such that every subgraph $\Gamma \subseteq G$ which is a
non-Hamiltonian $(n/5,2)$-expander with $\frac{an\ln n}{2\ln \ln
n}\leq |E(\Gamma)|\leq \frac{100an\ln n}{\ln \ln n}$ satisfies
$|E(G)\cap {\mathcal B}_{\Gamma}|> \frac{n^2p}{100}$.
\end{lemma}

\textbf{Proof.} First, notice that any $(n/5,2)$-expander is
connected. Indeed, let $C$ be a connected component of $G$. If
$|C|\leq n/5$ then clearly $C$ has neighbors outside, a
contradiction. Otherwise, since $G$ is an $(n/5,2)$-expander, $C$
must be of size at least $3n/5>n/2$. Hence there is exactly one such
component and $G$ is connected. Now, fix a non-Hamiltonian
$(n/5,2)$-expander $\Gamma$ in the complete graph $K_n$. Then
clearly $\Pr\left(\Gamma \subseteq G\right)=p^{|E(\Gamma)|}$. By
definition, the set of boosters of $\Gamma$, ${\mathcal
B}_{\Gamma}$, is a subset of the potential edges of $G$. Therefore,
$|E(G)\cap {\mathcal B}_{\Gamma}|\sim \textrm{Bin}(|{\mathcal
B}_{\Gamma}|,p)$ and the expected number of boosters is $|{\mathcal
B}_{\Gamma}|p\geq\frac{n^2p}{50}$ by Lemma~\ref{l:expander_booster}.
Now, by Lemma~\ref{Che} we get that $\Pr(|E(G)\cap {\mathcal
B}_{\Gamma}|\leq\frac{n^2p}{100})\leq \exp(-\frac{n^2p}{8})$.
Running over all choices of $\Gamma$ with $\frac{an\ln n}{2\ln \ln
n}\leq |E(\Gamma)|\leq \frac{100an\ln n}{\ln \ln n}$ and using the
union bound we get

\begin{align*}
\Pr&\left(\exists \textrm{ }\Gamma \textrm{ such that }\Gamma
\subseteq G \textrm{ and } |E(G)\cap {\mathcal B}_{\Gamma}|\leq
\frac{n^2p}{100}\right)\nonumber \\
\leq & \sum_{m=\frac{an\ln n}{2\ln\ln n}}^{\frac{100an\ln n}{\ln\ln
n}}\binom{\binom{n}{2}}{m}p^m \exp(-\frac{n^2p}{8})\nonumber \\
\leq & \sum_{m=\frac{an\ln n}{2\ln\ln n}}^{\frac{100an\ln n}{\ln\ln
n}}\left(\frac{en^2p}{2m}\right)^m \exp(-\frac{n^2p}{8})\nonumber \\
\leq &\sum_{m=\frac{an\ln n}{2\ln\ln n}}^{\frac{100an\ln n}{\ln\ln
n}}\exp\left(m\ln\left(\frac{en^2p}{2m}\right)-\frac{n^2p}{8}\right)=\heartsuit
\nonumber
\end{align*}
To complete the proof we should show that $\heartsuit=o(1)$. For
that goal we consider each of the cases $np=\omega(\ln^2 n)$ and
$np=O(\ln^2 n)$ separately. For the former we have that
$$\heartsuit\leq \sum_{m=\frac{an\ln n}{2\ln\ln n}}^{\frac{100an\ln n}{\ln\ln n}}\exp\left(n\ln^2 n-\frac{n^2p}{8}\right)=o(1);$$
and for the latter we have
\begin{align*}
\heartsuit\leq & \sum_{m=\frac{an\ln n}{\ln\ln n}}^{\frac{100an\ln
n}{\ln\ln n}}\exp\left(\frac{100an\ln n}{\ln \ln
n}\ln\left(\frac{enp\ln\ln n}{a\ln
n}\right)-\frac{n^2p}{8}\right)\\
\leq & \sum_{m=\frac{an\ln n}{2\ln\ln n}}^{\frac{100an\ln n}{\ln\ln
n}}\exp\left(\frac{100an\ln n}{\ln \ln n}\ln\left(C\ln n\ln\ln
n\right)-\frac{n^2p}{8}\right)\\
=& \sum_{m=\frac{an\ln n}{2\ln\ln n}}^{\frac{100an\ln n}{\ln\ln
n}}\exp\left((1+o(1))100an\ln n-\frac{n^2p}{8}\right)=o(1)
\end{align*}
This completes the proof. {\hfill $\Box$\medskip}

The following lemma shows that an $(R,c)$-expander with the
appropriate parameters is also $k$-vertex-connected.

\begin{lemma}[\cite{BFHK}, Lemma 5.1]\label{kconnectivity} For every positive integer $k$, if $G = (V,E)$ is an $(R,c)$-expander
such that $c \geq k$, and $Rc \geq \frac{1}{2}(|V|+k)$, then $G$ is
$k$-vertex-connected.
\end{lemma}

\subsection{Properties of $G \sim \gnp$}

Throughout this paper we use the following properties of $G
\sim\gnp$:

\begin{theorem} \label{PropertiesGnpDense}
Let $p \geq \frac{\ln n}{n}$ and recall our notation
$f(n):=\frac{np}{\ln n}$. A random graph $G \sim \gnp$ is typically
such that the following properties hold:
\begin{enumerate}[(P1)]

\item 
For every $v\in V$, $d(v) \leq 4np$. For every $\alpha >0$ there
are only $o(n)$ vertices with degree at least $(1+\alpha)np$.\\
If $f(n) = \omega(1)$ then for every $0 < \alpha < 1$ and for every
$v\in V$, $$(1-\alpha)np \leq d(v) \leq (1+\alpha)np.$$

\item 
For every subset $U\subseteq V$, $e(U)\leq \max\{3|U|\ln n,
3|U|^2p\}$ .

\item 
For every subset $U\subseteq V$ of size $|U|\leq \frac{n\ln\ln
n}{\ln n}$, $e(U)\leq 100|U|f(n)\ln\ln n$.

\item 
Let $\eps > 0$. For every constant $\alpha>0$ and for every subset
$U\subseteq V$ where $1\leq |U| \leq \frac{\alpha}{p}$, $|N(U)| \geq
\beta|U|np$, for $\beta =
\frac{1-\sqrt{\frac{(2+\eps)(\alpha+1)}{f(n)}}}{\alpha+1}$.

\item 
For every $U\subseteq V$, $\frac{1}{p}\leq |U|\leq \frac{n}{\ln n}$,
$|N(U)|\geq n/4$.

\item 
Let $\eps > 0$. For every $\alpha \geq \sqrt{\frac{4}{f(n)}}+\eps$
and for every set $U \subseteq V$, the number of edges between the
set and its complement $U^c$ satisfies:
\[e(U,U^c)\geq (1-\alpha)|U|(n-|U|)p.\]

\item 
Let $\eps,\alpha$ be two positive constants which satisfy
$\alpha^2\eps f(n) > 4$, and denote $m:=\frac{\eps n\ln\ln n}{\ln
n}$. For every two disjoint subsets $A,B\subseteq V$ with
$|A|=|B|=m$, $e(A,B) \geq (1-\alpha)m^2p$.

\item 
$e(A,B)\geq (1-\alpha)|A||B|p$ for every two disjoint subsets
$A,B\subseteq V$ with $|A|=\frac{10000n}{\ln\ln n}$, $|B|=n/10$ and
for every $\alpha>0$.

\item 
For every subset $U\subseteq V$ such that $1\leq |U| \leq
\frac{n}{\ln^2 n}$, and for every $\eps > 0$, $|\{v \in V\setminus
U: d(v,U)\leq \frac{\eps np}{\ln n}\}|=(1-o(1))n$.
\end{enumerate}
\end{theorem}
\textbf{Proof.} For the proofs of $\emph{(P4),(P5)}$ below we will
use the following:

Let $U\subseteq V$. For every vertex $v\in V\setminus U$ we have
that $\Pr\left(v\in N(U)\right)=1-(1-p)^{|U|}$ independently of all
other vertices. Therefore $|N(U)|\sim Bin(n-|U|,1-(1-p)^{|U|})$.
Notice that for any $0 < p < 1$ (all the properties above trivially
hold for $p=1$) and for any positive integer $k$ we have the
following variation of Bernoulli's inequality: $(1-p)^{-k} \geq
1+kp$. Therefore, $(1-(1-p)^{|U|}) \geq (1-\frac{1}{1+|U|p}) =
\frac{|U|p}{1+|U|p}$. It follows that:

\begin{equation}\label{ExpNU}
\Exp(|N(U)|) = (n-|U|)(1-(1-p)^{|U|}) \geq
\frac{(n-|U|)|U|p}{1+|U|p}.
\end{equation}

\begin{enumerate}
%
\item [\emph{(P1)}]
For every $v\in V$, since $d(v)\sim Bin(n-1,p)$ , it follows by
Lemma~\ref{l:Che} that $$\Pr(d(v) \geq 4np) \leq
\left(\frac{enp}{4np}\right)^{4np} < e^{-1.2np} \leq e^{-1.2\ln n} =
n^{-1.2}.$$ Applying the union bound we get that $$\Pr(\exists v\in
V \textrm{ with } d(v)\geq 4np)\leq n\cdot n^{-1.2}=o(1).$$

Now let $\alpha > 0$. By Lemma~\ref{Che} we get that for every $v
\in V$: $$\Pr(d(v) > (1+\alpha)np) \leq \exp(-\alpha'np) \leq
n^{-\alpha'},$$ for some constant $\alpha'$. Denote by $S$ the set
of all vertices with such degree. $\Exp(|S|) \leq n^{1-\alpha'}$.
$|S|$ is a nonnegative random variable, so by Markov's inequality we
get that:
$$\Pr(|S| > n^{1-\frac{\alpha'}{2}}) \leq
\frac{n^{1-\alpha'}}{n^{1-\frac{\alpha'}{2}}} =
n^{-\frac{\alpha'}{2}} = o(1).$$ Therefore, w.h.p. $|S| \leq
n^{1-\frac{\alpha'}{2}} = o(n)$.

Assume now that $f(n) = \omega(1)$, and let $0 < \alpha < 1$ be a
constant. By Lemma~\ref{Che} and the union bound we get that
$$\Pr(\exists v\in V \textrm{ with } d(v)\geq (1+\alpha)np)\leq
n\exp\left(-\frac{\alpha^2}{3}np\right)$$
$$= n\exp\left(-\frac{\alpha^2}{3}f(n)\ln n\right) = n^{-\omega(1)} =
o(1).$$ The lower bound is achieved in a similar way.

\item [\emph{(P2)}]
Since $e(U) \sim Bin\left(\binom{|U|}{2},p\right)$, using Lemma
\ref{l:Che} and the union bound we get that:
\begin{align*}
\Pr&\left(\exists U\subseteq V \textrm{ with } e(U)>\max\{3|U|\ln n,
3|U|^2p\}\right)\nonumber \\
\leq &\sum_{t=1}^{\frac{\ln n}{p}}
\binom{n}{t}\left(\frac{e\binom{t}{2}p}{3t\ln n}\right)^{3t\ln n} +
\sum_{t=\frac{\ln n}{p}}^{n} \binom{n}{t}
\left(\frac{e\binom{t}{2}p}{3t^2p}\right)^{3t^2p}\nonumber \\
\leq&\sum_{t=1}^{\frac{\ln n}{p}} \left[n\left(\frac{tp}{2\ln n}
\right)^{3\ln n}\right]^t + \sum_{t=\frac{\ln n}{p}}^{n}
\left[n\left(\frac{1}{2}\right)^{3tp}\right]^t \nonumber \\
\leq &\sum_{t=1}^{n}\left(\frac{e}{8}\right)^{t\ln n}
\leq\sum_{t=1}^{n} n^{-t} = o(1).\nonumber
\end{align*}
\item [\emph{(P3)}]
Let $U\subset V$ be a subset of size at most $\frac{n\ln\ln n}{\ln
n}$. Since $e(U)\sim \textrm{Bin}(\binom{|U|}{2},p)$, by Lemma
\ref{l:Che} we get that
$$\Pr(e(U)\geq 10|U|f(n)\ln\ln n)\leq \left(\frac{e|U|^2p}
{20|U|f(n)\ln\ln n}\right)^{10|U|f(n)\ln\ln n}.$$ Applying the union
bound we get that $$\Pr\left(\exists \textrm{ } U \textrm{ such that
} |U|\leq \frac{n\ln\ln n}{\ln n} \textrm{ with } e(U)\geq
10|U|f(n)\ln\ln n\right)$$
\begin{align*}
\leq & \sum_{k=1}^{\frac{n\ln\ln n}{\ln
n}}\binom{n}{k}\left(\frac{ek^2p} {20kf(n)\ln\ln
n}\right)^{10kf(n)\ln\ln n}\nonumber\\
\leq & \sum_{k=1}^{\frac{n\ln\ln n}{\ln n}}
\left[\frac{en}{k}\left(\frac{ekp}{20f(n)\ln\ln n
}\right)^{10f(n)\ln\ln n}\right]^k\nonumber\\
= & \sum_{k=1}^{\frac{n\ln\ln n}{\ln n}}
\left[\frac{e^2np}{20f(n)\ln\ln n}\left(\frac{ekp}{20f(n)\ln\ln
n}\right)^{10f(n)\ln\ln n-1}\right]^k\nonumber\\
\leq & \sum_{k=1}^{\frac{n\ln\ln n}{\ln n}}\left[\frac{e^2\ln
n}{20\ln\ln n} \left(\frac{enp\ln\ln n}{20f(n)\ln n\ln\ln
n}\right)^{10f(n)\ln\ln n-1}\right]^k\nonumber\\
\leq & \sum_{k=1}^{\frac{n\ln\ln n}{\ln n}}\left[\frac{e^2\ln
n}{20\ln\ln n} \left(\frac{e}{20}\right)^{10f(n)\ln\ln
n-1}\right]^k\nonumber\\
= & o(1).\nonumber
\end{align*}

\item [\emph{(P4)}]
Since $n-|U| = (1-o(1))n$ in this range, by \eqref{ExpNU} we have
that:
$$\Exp(|N(U)|) \geq \frac{(n-|U|)|U|p}{1+|U|p} \geq
(1-o(1))\frac{|U|np}{\alpha+1}.$$

By Lemma~\ref{Che} we have that for any $\delta >0$:
$$\Pr\left(|N(U)| < (1-\delta)\Exp(|N(U)|)\right) \leq
e^{-\frac{\delta^2}{2}\Exp(|N(U)|)} \leq e^{-\alpha'|U|np},$$ where
$\alpha'=\frac{\delta^2}{(2+o(1))(\alpha+1)}$. Now, by taking
$\delta=\sqrt{\frac{(2+\eps)(\alpha+1)}{f(n)}}$ (for some $\eps>0$)
we get that $\alpha'f(n) > 1+\frac{\eps}{3}$, and so by applying the
union bound we get that:
$$\Pr(\exists \textrm{ such } U) \leq
\sum_{k=1}^{\alpha/p}\binom{n}{k} e^{-\alpha'knp} \leq
\sum_{k=1}^{\alpha/p} \left[ne^{-\alpha'f(n)\ln n}\right]^k =
o(1).$$

Therefore, w.h.p. for every such $U$, $|N(U)| \geq
(1-\delta)\Exp(|N(U)|) \geq \beta|U|np$, for $\beta =
\frac{1-\sqrt{\frac{(2+\eps)(\alpha+1)}{f(n)}}}{\alpha+1}$.

\item [\emph{(P5)}]
Let $\frac{1}{p} \leq |U| \leq \frac{n}{\ln n}$. By \eqref{ExpNU},
$\Exp\left(|N(U)|\right) \geq \frac{(n-|U|)|U|p}{1+|U|p}\geq n/3$.

By Lemma~\ref{Che} we have that $\Pr\left(|N(U)|\leq n/4\right)\leq
e^{-0.01n}$.

Applying the union bound we get that
$$\Pr\left(\exists \textrm{ such }
U\right) \leq \sum_{k=1/p}^{n/\ln n}\binom{n}{k} e^{-0.01n} \leq
n\binom{n}{\frac{n}{\ln n}}e^{-0.01n}$$
$$\leq n(e\ln n)^{\frac{n}{\ln n}}e^{-0.01n} = n\exp\left(\frac{n}{\ln n}\ln(e\ln
n)-0.01n\right) = o(1).$$

\item [\emph{(P6)}]
Assume first that $|U| \leq n/2$, otherwise switch the roles of $U$
and $U^c$. Since every edge between $U$ and $U^c$ is chosen
independently, $e(U,U^c) \sim Bin(|U||U^c|,p)$. By Lemma~\ref{Che}
we have that for given $\alpha>0$ and $U\subseteq V$:
$$\Pr\left(e(U,U^c) < (1-\alpha)|U|(n-|U|)p\right) \leq
\exp\left(-\frac{\alpha^2}{2}|U|(n-|U|)p\right)$$
$$\leq \exp\left(-\frac{\alpha^2}{4}|U|np\right) \leq
\exp\left(-(\frac{1}{f(n)}+\delta)|U|np\right)$$
$$= \exp\left(-|U|(\ln n + \delta np)\right),$$ for some
$\delta=\delta(\eps) >0$. By the union bound we get that:
$$\Pr(\exists \textrm{ such } U) \leq \sum_{k=1}^{n/2} \binom{n}{k}
\exp\left(-k(\ln n + \delta np)\right)\leq \sum_{k=1}^{n/2}
\left[n\exp\left(-\ln n-\delta np\right)\right]^k $$
$$= \sum_{k=1}^{n/2}\left(n^{-\delta f(n)}\right)^k=o(1).$$

\item [\emph{(P7)}]
Similarly to $\emph{(P6)}$, given $A,B \subset V$, $|A|=|B|=m$,
$e(A,B)\sim Bin(m^2,p)$. Therefore, by Lemma~\ref{Che} we have that:
$$\Pr\left(e(A,B)\leq (1-\alpha)m^2p\right)\leq \exp\left(-\frac{\alpha^2}{2}m^2p\right).$$
Applying the union bound we get that:
$$\Pr\left(\exists \textrm{ such } A,B\right)\leq
\binom{n}{m}^2\exp\left(-\frac{\alpha^2}{2}m^2p\right)\leq
\left[\left(\frac{en}{m}\right)^2\exp\left(-\frac{\alpha^2}{2}mp\right)\right]^m$$
$$=\left[\left(\frac{e\ln n}{\eps\ln\ln
n}\right)^2\exp\left(-\frac{\alpha^2}{2}\eps f(n)\ln\ln
n\right)\right]^m \leq \left[(\ln n)^2(\ln
n)^{-(2+\delta)}\right]^m=o(1),$$ for some
$\delta=\delta(\eps,\alpha)>0$.

\item [\emph{(P8)}]
Given subsets $A,B\subseteq V$ as described, since $e(A,B)\sim
\textrm{Bin}(|A||B|,p)$, by Lemma~\ref{Che} we get that
$$\Pr\left(e(A,B)\leq (1-\alpha)|A||B|p \right) \leq
\exp\left(-\frac{\alpha^2}{2}|A||B|p\right) =
\exp\left(-\frac{\alpha' n^2p}{\ln\ln n}\right),$$

for some constant $\alpha'$. Applying the union bound we get that:
$$\Pr\left(\exists \textrm{ such } A,B\right)\leq
\binom{n}{\frac{1000n}{\ln \ln n}}
\binom{n}{n/10}\exp\left(-\frac{\alpha' n^2p}{\ln\ln
n}\right)\leq4^n\exp\left(-\omega(n)\right) = o(1).$$

\item [\emph{(P9)}]
Assume towards a contradiction that there exists a subset
$U\subseteq V$ such that $1\leq |U| \leq \frac{n}{\ln^2 n}$ and that
there are $\Theta(n)$ vertices $v\in V\setminus U$ with $d(v,U)\geq
\frac{\eps np}{\ln n}$. Therefore, the average degree of the
vertices in $U$ is at least $\Theta\left(n\frac{np}{\ln n
}\frac{1}{|U|}\right) = \Omega\left(np\ln n\right)$. But by
$\emph{(P1)}$, $d(v)\leq 4np$ for every $v\in V$ --- a
contradiction. Hence, $|\{v \in V\setminus U: d(v,U)\leq \frac{\eps
np}{\ln n}\}|=o(n)$. {\hfill $\Box$\medskip}
\end{enumerate}

The following two lemmas may seem somewhat unnatural, but they will
be crucial for our purposes. The first one will be useful in the
proof of Theorem~\ref{EnforcerWin}:

\begin{lemma} \label{AtLeast80}
Let $p \geq \frac{80\ln n}{n}$. A random graph $G \sim \gnp$ is
typically such that for every set $U \subseteq V$ of size
$\frac{80}{p} \leq |U| \leq \frac{n}{\ln n}$, and for every set $W
\subseteq N(U)$ of size $|W| = \frac{1}{2}|N(U)|$, the following
holds:
$$e(U,W) \geq \frac{1}{50}|U|np.$$
\end{lemma}
\begin{proof}
First, we prove the following claim:
\begin{claim}\label{claim1}
For $p \geq \frac{80\ln n}{n}$, $G \sim \gnp$ is typically such that
for every two disjoint sets $U,W \subseteq V$ such that
$\frac{80}{p} \leq |U| \leq \frac{n}{\ln n}$ and $|W| = \frac
{n}{100}$, $e(U,W) \leq 1.5|U||W|p$.
\end{claim}

\textbf{Proof of Claim~\ref{claim1}.} Let $U,W \subseteq V$ as
described above. Since $e(U,W) \sim Bin(|U||W|,p)$, by Lemma
\ref{Che} we have that:
$$\Pr(e(U,W) > 1.5|U|\frac{n}{100}p) \leq
e^{-\frac{1}{1200}|U|np}.$$ By the union bound we get that:
$$\Pr(\exists \textrm{ such } U,W) \leq
\sum_{k=\frac{80}{p}}^{\frac{n}{\ln n}} \binom{n}{k}
\binom{n}{n/100} e^{-\frac{1}{1200}knp}$$
$$\leq\sum_{k=\frac{80}{p}}^{\frac{n}{\ln n}}
\left[\left(\frac{en}{k}\right)(100e)^{n/100k}e^{-np/1200}\right]^k
\leq \sum_{k=\frac{80}{p}}^{\frac{n}{\ln n}}
\left[\left(\frac{enp}{80}\right)(e^6)^{np/8000}e^{-np/1200}\right]^k$$
$$\leq\sum_{k=\frac{80}{p}}^{\frac{n}{\ln n}}
\left[np\exp(-\frac{np}{12000})\right]^k = o(1).$$ {\hfill $\Box$
\medskip}

Now we return to the proof of Lemma~\ref{AtLeast80}, and we assume
that $G$ satisfies the properties of
Theorem~\ref{PropertiesGnpDense} and Claim~\ref{claim1}. Let $U
\subseteq V$, $\frac{80}{p} \leq |U| \leq \frac{n}{\ln n}$, fix some
$W \subseteq N(U)$ such that $|W| = \frac{1}{2}|N(U)|$, and denote
$W'=N(U)\setminus W$. Denote by $E_1$ the number of edges between
$U$ and $W$, and by $E_2$ the number of edges between $U$ and $W'$.
Notice that for these sizes of $U$, $n-|U| = (1-o(1))n$, and that
$\sqrt{\frac{4}{80}} < 0.23$, so by $\emph{(P6)}$ of
Theorem~\ref{PropertiesGnpDense} we have that $E_1 + E_2 \geq
0.77|U|np$.

Now partition $W'$ into subsets of size $\frac{n}{100}$ (the last
one may be of smaller size). Since $|W'| \leq \frac{n}{2}$, 50 such
subsets suffice. By Claim~\ref{claim1}, there are at most
$\frac{1.5}{100}|U|np$ edges between $U$ and each of the subsets, so
$E_2 \leq 0.75|U|np$. Putting the two inequalities together, we get
that $E_1 \geq \frac{1}{50}|U|np$.
\end{proof}\\

The second lemma will be a key ingredient in the main proof of the
next subsection.

\begin{lemma} \label{HarmonicDeg}
Let $p=\omega(\frac{\ln n}{n})$. Then $G\sim G(n,p)$ is typically
such that the following holds:
\begin{description}
\item For every subset $J_N=\{v_1,\ldots ,v_{N}\}\subseteq V$ we
have that
$$\sum_{j=1}^N\frac{e(v_j,J_j)}{j}= o(np)$$ where $N=\frac{n}{\ln^3 n}$
and $J_j=\{v_1,\ldots,v_j\}$, $1\leq j\leq N$.
\end{description}
\end{lemma}

\textbf{Proof.} Let $t=\lceil \log_2 (N+1)\rceil$. Partition
$J_N=I_0\cup I_1\cup \ldots \cup I_{t}$ in such a way that
$I_i=\{v_{2^i},\ldots v_{2^{i+1}-1}\}$ for every $0\leq i<t$, and
$I_{t}=\{v_{2^t},\ldots,v_N\}$. Notice that $|I_i|=2^i$ for every
$1\leq i<t$. We have:
$$\sum_{j=1}^N\frac{e(v_j,J_j)}{j}\leq
\sum_{i=1}^{t}\frac{e(I_i) +
e(I_i,J_{2^i})}{2^i}\leq\sum_{i=1}^{t}\frac{e(J_{2^{i+1}})}{2^i}=\spadesuit.$$
Now we distinguish between the following two cases:
\begin{enumerate} [$(i)$]
\item $pn=\omega(\ln^2 n)$. In this case we have by property
$(P2)$ of Theorem~\ref{PropertiesGnpDense}:
$$\spadesuit \leq \sum_{i=1}^{\log_2(\frac{\ln n}{p})}\frac{3|J_{2^{i+1}}|\ln
n}{2^i}+\sum_{i=\log_2(\frac{\ln n}{p})}^{t} \frac{3|J_{2^{i+1}}|^2
p}{2^i}$$
$$\leq c_1\ln n \ln(\frac{\ln n}{p}) + c_2Np\leq c_1\ln^2n+c_2Np,$$
for some positive constants $c_1$ and $c_2$. This is clearly $o(np)$
as desired.

\item $pn=O(\ln^2 n)$. In this case we need a more careful calculation. First, we
prove the following claim:

\begin{claim} \label{claim2}
If $np=O(\ln^2 n)$, then for every $c>3$, $G\sim G(n,p)$ is
typically such that $e(X)\leq c|X|$ for every subset $X\subseteq V$
of size $|X|\leq \frac{n}{\ln^3 n}$.
\end{claim}

\textbf{Proof of Claim~\ref{claim2}.} Let $X\subset V$ be a subset
of size at most $\frac{n}{\ln^3 n}$. Since $e(X)\sim
\textrm{Bin}(\binom{|X|}{2},p)$, by Lemma~\ref{l:Che} we get that
$\Pr(e(X)\geq c|X|)\leq \left(\frac{e|X|^2p}{c|X|}\right)^{c|X|}$.
Applying the union bound we get that $$\Pr\left(\exists X \textrm{
such that } |X|\leq \frac{n}{\ln^3 n} \textrm{ and } e(X)\geq
c|X|\right)$$
$$\leq \sum_{k=1}^{\frac{n}{\ln^3 n}}\binom{n}{k}\left(\frac{ek^2p}{c
k}\right)^{c k}\leq \sum_{k=1}^{\frac{n}{\ln^3 n}}
\left[\frac{en}{k}\left(\frac{ekp}{c }\right)^{c}\right]^k$$
$$= \sum_{k=1}^{\frac{n}{\ln^3 n}}
\left[\frac{e^2np}{c}\left(\frac{ekp}{c}\right)^{c-1}\right]^k \leq
\sum_{k=1}^{\frac{n}{\ln^3
n}}\left[O(\ln^2n)O(\ln^{-1}n)^{c-1}\right]^k $$
$$= \sum_{k=1}^{\frac{n}{\ln^3 n}} \left[O(\ln n)^{3-c}\right]^k =o(1).$$
{\hfill $\Box$ \medskip}

Now, applying Claim~\ref{claim2} with $c=4$, we get that
$$\spadesuit \leq \sum_{i=1}^{t}\frac{4|J_{2^{i+1}}|}{2^i} \leq
\sum_{i=1}^{t}8=O(\ln n)=o(np).$$
\end{enumerate}

 This completes the proof of Lemma~\ref{HarmonicDeg}. {\hfill $\Box$ \medskip}

\subsection{The minimum degree game}

In the proof of Theorem~\ref{mainDense}, Maker has to build a
suitable expander which possesses some relevant properties. The
first step towards creating a good expander is to create a spanning
subgraph with a large enough minimum degree. The following theorem
was proved in \cite{GS}:

\begin{theorem} [\textbf{\cite{GS}, Theorem 1.3}] \label{mindegreeKn}
Let $\varepsilon>0$ be a constant. Maker has a strategy to build a
graph with minimum degree at least
$\frac{\varepsilon}{3(1-\varepsilon)}\ln n$ while playing against
Breaker's bias of $(1-\varepsilon)\frac{n}{\ln n}$ on $E(K_n)$.
\end{theorem}

In fact, the following theorem can be derived immediately from the
proof of Theorem~\ref{mindegreeKn}:

\begin{theorem} \label{mindegreeKnimproved}
Let $\varepsilon>0$ be a constant. Maker has a strategy to build a
graph with minimum degree at least
$c=c(n)=\frac{\varepsilon}{3(1-\varepsilon)}\ln n$ while playing
against a Breaker's bias of $(1-\varepsilon)\frac{n}{\ln n}$ on
$E(K_n)$. Moreover, Maker can do so within $cn$ moves and in such a
way that for every vertex $v\in V(K_n)$, at the same moment $v$
becomes of degree $c$ in Maker's graph, there are still $\Theta(n)$
free edges incident with $v$.
\end{theorem}

Using Theorem~\ref{mindegreeKnimproved}, the third author of this
paper proved in \cite{K} that Maker has a strategy to build a good
expander. Here, we wish to prove an analog of Theorem
\ref{mindegreeKnimproved} for $\gnp$:

\begin{theorem} \label{minimumdegree}
\label{strategys} Let $p=\omega\left(\frac{\ln n}{n}\right)$,
$\varepsilon>0$ and let $b=(1-\eps)\frac{np}{\ln n}$. Then $G\sim
\gnp$ is typically such that in the $(1,b)$ Maker-Breaker game
played on $E(G)$, Maker has a strategy to build a graph with minimum
degree $c=c(n)=\frac{\varepsilon}{6} \ln n$. Moreover, Maker can do
so within $cn$ moves and in such a way that for every vertex $v\in
V(G)$, at the same moment that $v$ becomes of degree $c$ in Maker's
graph, at least $\eps np/3$ edges incident with $v$ are free.
\end{theorem}

\textbf{Proof of Theorem~\ref{minimumdegree}.} The proof is very
similar to the proof of Theorem~\ref{mindegreeKn} so we omit some of
the calculations (for more details, the reader is referred to
\cite{GS}). Since claiming an extra edge is never a disadvantage for
any of the players, we can assume that Breaker is the first player
to move. A vertex $v\in V$ is called \emph{dangerous} if $d_M(v)<
c$. The game ends at the first moment in which either none of the
vertices is dangerous (and Maker won), or there exists a dangerous
vertex $v\in V$ with less than $\varepsilon np/3$ free edges
incident to it (and Breaker won). For every vertex $v\in V$ let
$\danger(v) := d_B(v) - 2b \cdot d_M(v)$ be the \emph{danger value}
of $v$. For a subset $X\subseteq V$, we define
$\avdanan(X)=\frac{\sum_{v\in X} \danger(v)}{|X|}$ (the average
danger of vertices in $X$).

The strategy proposed to Maker is the following one:

\textbf{Maker's strategy $S_M$:} As long as there is a vertex of
degree less than $c$ in Maker's graph, Maker claims a free edge $vu$
for some $v$ which satisfies $\danger(v)=\max\{\danger(u): u\in V\}$
(ties are broken arbitrarily).

Suppose towards a contradiction that Breaker has a strategy $S_B$ to
win against Maker who plays according to the strategy $S_M$ as
suggested above. Let $g$ be the length of this game and let
$I=\{v_1,\ldots,v_g\}$ be the multi-set which defines Maker's game,
i.e, in his $i$th move, Maker plays at $v_i$ (in fact, according to
the assumption Maker does not make his $g$th move, so let $v_g$ be
the vertex which made him lose). For every $0\leq i\leq g-1$, let
$I_i=\{v_{g-i},\ldots,v_g\}$. Following the notation of \cite{GS},
let $\danger_{B_i}(v)$ and $\danger_{M_i}(v)$ denote the danger
value of a vertex $v\in V$ \emph{directly before} Breaker's and
Maker's $i$th move, respectively. Notice that in his last move,
Breaker claims $b$ edges to decrease the minimum degree of the free
graph to at most $\varepsilon np/3$. In order to be able to do that,
directly before Breaker's last move $B_g$, there must be a dangerous
vertex $v_g$ with $d_M(v_g) \leq c-1$ and $d_F(v_g) \leq \frac{\eps
np}{3}+b$. By $(P1)$ of Theorem~\ref{PropertiesGnpDense} we can
assume that $\delta(G) \geq (1-\frac{\eps}{12})np$. Therefore we
have that $\danger_{B_g}(v_g) \geq (1 -
\frac{\eps}{12}-\frac{\eps}{3})np - b -
2b(c-1)=(1-\frac{5}{12}\eps)np-b(2c-1)\geq(1-\frac{3}{4}\eps)np$.

Analogously to the proof of Theorem~\ref{mindegreeKn} in \cite{GS},
we state the following lemmas which estimate the change of the
average danger after each move of any of the players. In the first
lemma we estimate the change after Maker's move:

\begin{lemma} \label{lem:makermindeg}
Let $i$, $1\leq i \leq g-1$,

$(i)$ if $I_i\neq I_{i-1}$, then $\avdanan_{M_{g-i}}(I_{i}) -
\avdanan_{B_{g-i+1}}(I_{i-1}) \geq 0.$

$(ii)$ if $I_i=I_{i-1}$, then $\avdanan_{M_{g-i}}(I_{i}) -
\avdanan_{B_{g-i+1}}(I_{i-1}) \geq \frac{2b}{|I_i|}.$
\end{lemma}

In the second lemma we estimate the change of the average danger
during Breaker's moves:

\begin{lemma}\label{lem:breakermindeg}
Let $i$ be an integer, $1\leq i\leq g-1$.

$(i)$ $\avdanan_{M_{g-i}}(I_i) - \avdanan_{B_{g-i}}(I_i) \leq
\frac{2b}{|I_i|}$

$(ii)$ $\avdanan_{M_{g-i}}(I_i) - \avdanan_{B_{g-i}}(I_i) \leq
\frac{b+e(v_{g-i},I_i)+a(i-1)-a(i)}{|I_i|}$, where $a(i)$ denotes
the number of edges spanned by $I_i$ which Breaker claimed in the
first $g-i-1$ rounds.
\end{lemma}

Combining Lemmas~\ref{lem:makermindeg} and~\ref{lem:breakermindeg},
we get the following corollary which estimates the change of the
average danger after a full round:

\begin{corollary}\label{coro:danger-change}
Let $i$ be an integer, $1\leq i\leq g-1$.

$(i)$ if $I_i=I_{i-1}$, then $\avdanan_{B_{g-i}}(I_{i}) -
\avdanan_{B_{g-i+1}}(I_{i-1}) \geq 0.$

$(ii)$ if $I_i\neq I_{i-1}$, then $\avdanan_{B_{g-i}}(I_{i}) -
\avdanan_{B_{g-i+1}}(I_{i-1}) \geq -\frac{2b}{|I_i|}$

$(iii)$ if $I_i\neq I_{i-1}$, then $\avdanan_{B_{g-i}}(I_{i}) -
\avdanan_{B_{g-i+1}}(I_{i-1}) \geq
-\frac{b+e(v_{g-i},I_i)+a(i-1)-a(i)}{|I_i|}$, where $a(i)$ denotes
the number of edges spanned by $I_i$ which Breaker took in the first
$g-i-1$ rounds.
\end{corollary}

In order to complete the proof, we prove that before Breaker's first
move, $\avdanan_{B_1}(I_{g-1})>0$, thus obtaining a contradiction.

Let $N:=\frac{n}{\ln^3 n}$. For the analysis, we split the game into
two parts: the main game, and the end game which starts when
$|I_{i}| \leq N$.

Let $|I_g|=r$ and let $i_1 < \ldots < i_{r-1}$ be those indices for
which $I_{i_j}\neq I_{i_j-1}$. Note that $|I_{i_j}|=j+1$. Note also
that since $I_{i_j-1} = I_{i_{j-1}}$ and $i_{j-1} \leq i_j-1$,
$a(i_j-1) \leq a(i_{j-1})$.

Recall that the danger value of $v_g$ directly before $B_g$ is at
least
\begin{equation}
\danger_{B_g}(v_g)>(1-\frac{3}{4}\eps)np.\label{dangBgI0}
\end{equation}

Assume first that $r < N$.


\begin{eqnarray}\label{eq:lastlinehugecalc}
\avdanan_{B_1}(I_{g-1}) & = & \avdanan_{B_g}(I_{0}) +
\sum_{i=1}^{g-1}
\left( \avdanan_{B_{g-i}}(I_{i}) - \avdanan_{B_{g-i+1}}(I_{i-1}) \right) \nonumber \\
& \geq & \avdanan_{B_g}(I_{0}) + \sum_{j=1}^{r-1} \left(
\avdanan_{B_{g-i_j}}(I_{i_j}) - \avdanan_{B_{g-i_j+1}}(I_{i_j-1})
\right)
\enspace\enspace \mbox{[by Corollary~\ref{coro:danger-change}$(i)$]} \nonumber \\
& \geq & \avdanan_{B_g}(I_{0}) - \sum_{j=1}^{r-1}
\frac{b+e(v_{g-i_j},I_{i_j})+a(i_j-1)-a(i_j)}{j+1}
\enspace\enspace \mbox{[by Corollary~\ref{coro:danger-change}$(iii)$]} \nonumber \\
& \geq & \avdanan_{B_g}(I_{0}) - b\ln
r-\sum_{j=1}^{r-1}\frac{e(v_{g-i_j},I_{i_j})}{j+1}-\frac{a(0)}{2}+\sum_{j=2}^{r-1}\frac{a(i_{j-1})}{(j+1)j}+\frac{a(i_{r-1})}{r}
\enspace\enspace \mbox{} \nonumber \\
& \geq & \avdanan_{B_g}(I_{0}) - b\ln r-o(np)
\enspace\enspace \mbox{[by Lemma ~\ref{HarmonicDeg}]} \nonumber \\
& > & (1-\frac{3}{4}\varepsilon)np-(1-\varepsilon+o(1))np
\enspace\enspace \mbox{[by~\eqref{dangBgI0}]}\nonumber\\
& > &0.
\end{eqnarray}

Assume now that $r \geq N$.
\begin{eqnarray}
\avdanan_{B_1}(I_{g-1}) & = & \avdanan_{B_g}(I_{0}) +
\sum_{i=1}^{g-1}
\left( \avdanan_{B_{g-i}}(I_{i}) - \avdanan_{B_{g-i+1}}(I_{i-1}) \right) \nonumber \\
& \geq & \avdanan_{B_g}(I_{0}) + \sum_{j=1}^{r-1} \left(
\avdanan_{B_{g-i_j}}(I_{i_j}) - \avdanan_{B_{g-i_j+1}}(I_{i_j-1})
\right)
\enspace\enspace \mbox{[by Corollary~\ref{coro:danger-change}$(i)$]} \nonumber \\
& = & \avdanan_{B_g}(I_{0}) + \sum_{j=1}^{N - 1} \left(
\avdanan_{B_{g-i_j}}(I_{i_j}) - \avdanan_{B_{g-i_j+1}}(I_{i_j-1})
\right)
+ \nonumber \\
&  & \sum_{j=N}^{r - 1}
\left( \avdanan_{B_{g-i_j}}(I_{i_j}) - \avdanan_{B_{g-i_j+1}}(I_{i_j-1}) \right) \nonumber \\
& \geq & \danger_{B_g}(v_g) - \sum_{j=1}^{N-1} \frac{b}{j+1}- o(np)
-\sum_{j=N}^{r - 1} \frac{2b}{j+1}
\enspace\enspace \mbox{[by Corollary~\ref{coro:danger-change}$(ii)$ and (\ref{eq:lastlinehugecalc})]} \nonumber \\
& \geq &   (1-\frac{3}{4}\varepsilon)np-b\ln n-o(np)-2b(\ln n-\ln
\frac{n}{\ln^3 n}) \enspace\enspace \mbox{[by~\eqref{dangBgI0}]}
\nonumber \\
& = & (1-\frac{3}{4}\varepsilon)np-(1-\eps)np-o(np)-6b\ln\ln n
\nonumber\\
& = & \frac{\eps}{4}np-o(np) \nonumber\\
& > & 0.\nonumber
\end{eqnarray}

This completes the proof.  {\hfill $\Box$ \medskip\\}

\section{Maker-Breaker games on $\gnp$}

\subsection{Breaker's win}
In this subsection we prove Theorem~\ref{BreakerIsolatingAvertex}.

Chv\'atal and Erd\H{o}s proved in \cite{CE} that playing on the edge
set of the complete graph $K_n$, if Breaker's bias is
$b=(1+\varepsilon)\frac{n}{\ln n}$, then Breaker is able to isolate
a vertex in his graph and thus to win a lot of natural games such as
the perfect matching game, the hamiltonicity game and the
$k$-connectivity game.

In their proof, Breaker wins by creating a large clique which is
disjoint of Maker's graph and then playing the box game on the stars
centered in this clique. Our proof is based on the same idea.

\textbf{Proof of Theorem~\ref{BreakerIsolatingAvertex}:} First, we
may assume that $p \geq \frac{\ln n}{n}$, since otherwise $G\sim
\gnp$ typically contains isolated vertices and Breaker wins no
matter how he plays. Now we introduce a strategy for Breaker and
then we prove it is a winning strategy. At any point during the
game, if Breaker cannot follow the proposed strategy then he
forfeits the game. Breaker's strategy is divided into the following
two stages:

\textbf{Stage I:} Throughout this stage Breaker maintains a subset
$C \subseteq V$ which satisfies the following properties:
\begin{description}
\item[$(i)$] $E_G(C)=E_B(C)$.
\item[$(ii)$] $d_M(v) = 0$ for every $v \in C$.
\item[$(iii)$] $d_G(v) \leq (1+\eps/2)np$ for every $v \in C$.
\end{description}
Initially, $C=\emptyset$. In every move, Breaker increases the size
of $C$ by at least one. This stage ends after the first move in
which $|C|\geq \frac{n}{\ln^2 n}$.

\textbf{Stage II:} For every $v\in C$, let $A_v=\{vu\in E(G):
vu\notin E(B)\}$. In this stage, Breaker claims all the elements of
one of these sets.

It is evident that if Breaker can follow the proposed strategy, then
he isolates a vertex in Maker's graph and wins the game. It thus
remains to prove that Breaker can follow the proposed strategy. We
consider each stage separately.

\textbf{Stage I:} Notice that in every move Maker can decrease the
size of $C$ by at most one. Hence, it is enough to prove that in
every move Breaker is able to find at least two vertices which are
isolated in Maker's graph and have bounded degree as required, and
to claim all the free edges between them and $C$. For this it is
enough to prove that Breaker can always find two vertices $u,v\in
V\setminus C$ which have the proper degree in $G$ and are isolated
in Maker's graph, and such that $e(v,C)$, $e(v,C)\leq \frac{b}{2}$.
Since this stage lasts $o(n)$ moves, and the number of vertices with
too high degree in $G$ is $o(n)$ by property $(P1)$ of
Theorem~\ref{PropertiesGnpDense}, the existence of such vertices is
trivial by property $(P9)$ of Theorem~\ref{PropertiesGnpDense}.

\textbf{Stage II:} Notice that $|C|\geq \frac{n}{\ln^2 n}$ and that
$A_v\cap A_u=\emptyset$ for every two vertices $u\neq v$ in $C$. In
addition, by the way Breaker chooses his vertices we have that
$|A_v|\leq (1+\varepsilon/2)np$ for every $v \in C$. Recall that
$b=(1+\varepsilon)\frac{np}{\ln n}>\frac{(1+\varepsilon/2)np}{\ln
|C|}$. Therefore, by Theorem~\ref{boxgame} Breaker (as BoxMaker)
wins the Box Game on these sets.

This completes the proof. {\hfill $\Box$\medskip\\}

\subsection{Maker's win}

In this subsection we prove Theorems~\ref{mainDense} and
\ref{mainSparse}. We start with providing Maker with a winning
strategy in the Hamiltonicity game for each case (which implies the
perfect matching game) and then we sketch the changes which need to
be done in order to turn it into a winning strategy in the
$k$-connectivity game as well.

\textbf{Proof of Theorem~\ref{mainDense}.} First we describe a
strategy for Maker and then prove it is a winning strategy.

At any point during the game, if Maker is unable to follow the
proposed strategy (including the time limits), then he forfeits the
game. Maker's strategy is divided into the following three stages:

\textbf{Stage I:} Maker builds an $(\frac{10000n}{\ln\ln
n},2)$-expander within $\frac{100n\ln n}{\ln \ln n}$ moves.

\textbf{Stage II:} Maker makes his graph an $(n/5,2)$-expander
within additional $\frac{100n\ln n}{\ln \ln n}$ moves.

\textbf{Stage III:} Maker makes his graph Hamiltonian by adding at
most $n$ boosters.

It is evident that if Maker can follow the proposed strategy without
forfeiting the game he wins. It thus suffices to prove that indeed
Maker can follow the proposed strategy. We consider each stage
separately.

\textbf{Stage I:} In his first $\frac{100n\ln n}{\ln \ln n}$ moves,
Maker creates a graph with minimum degree $c=c(n)=\frac{100\ln
n}{\ln \ln n}$. Maker plays according to the strategy proposed in
Theorem \ref{minimumdegree} except of the seemingly minor but
crucial change that in every move, when Maker needs to claim an edge
incident with a vertex $v$, Maker randomly chooses such a free edge.
We prove that, with a positive probability, this non-deterministic
strategy ensures that Maker's graph is an $(\frac{10000n}{\ln\ln
n},2)$ expander and then, since our game is a perfect information
game, we conclude that indeed there exists a deterministic such
strategy for Maker. Recall that according to the strategy proposed
in Theorem \ref{minimumdegree}, at any move Maker claims a free edge
$vu$ with $\danger(v)=\max\{\danger(u):u\in V\}$. In this case we
say that the edge $vu$ is \emph{chosen} by $v$. We wish to show that
the probability of having a subset $A\subseteq V$ with $|A|\leq
\frac{10000n}{\ln\ln n}$ and $|N_M(A)|\leq 2|A|-1$ is $o(1)$. To
that end, we can assume that $G$ satisfies all the properties listed
in Theorem~\ref{PropertiesGnpDense} and Theorem~\ref{minimumdegree}.

Assume that there exists a subset $A\subset V$ of size $|A|\leq
\frac{10000n}{\ln\ln n}$ such that after this stage $N_M(A)$ is
contained in a set $B$ of size at most $2|A|-1$. This implies that
\[|E_M(A, A\cup B)|\geq c|A|/2=\frac{50|A|\ln n}{\ln \ln n}.\]

%
%

Recall that $f(n):=\frac{np}{\ln n}$. We distinguish between the
following two cases:

\textbf{Case I:} At least $c|A|/4$ edges of Maker which are incident
to $A$ were chosen by vertices from $A$. Notice that if $|A|\leq
\frac{n\ln\ln n}{\ln n}$ there are at most $o(|A|)$ vertices $v\in
A$ such that $e(v,A\cup B)= \Omega(f(n)(\ln\ln n)^2)$, since
otherwise we have that $e(A\cup B)=\Omega(f(n)\ln\ln ^2 n|A|)$ which
contradicts $(P3)$ of Theorem~\ref{PropertiesGnpDense} and that if
$\frac{n\ln\ln n}{\ln n}<|A|\leq \frac{10000n}{\ln\ln n}$ then there
are at most $o(|A|)$ vertices $v\in A$ such that $e(v,A\cup B)=
\Omega(np)$ (follows from $(P2)$ of
Theorem~\ref{PropertiesGnpDense}). Consider an edge $e=ab$ with
$a\in A$ and $b\in A\cup B$ and assume that $e$ has been chosen by
$a$. Notice that by Theorem~\ref{minimumdegree}, when Maker chose
$e$, the vertex $a$ had at least $\varepsilon np/3$ free neighbors.
Therefore, for at least $(1-o(1))|A|$ such vertices $a\in A$, the
probability that Maker chose an edge with a second endpoint in
$A\cup B$ is at most $\left(\frac{f(n)(\ln \ln n)^2}{\varepsilon
np/3}\right)=\frac{3(\ln\ln n)^2}{\varepsilon \ln n}$ when $|A|\leq
\frac{n\ln\ln n}{\ln n}$ or an arbitrarily small constant $\delta>0$
when $\frac{n\ln\ln n}{\ln n}<|A|\leq \frac{10000n}{\ln\ln n}$.
Therefore, the probability that all of Maker's edges incident to $A$
were chosen in $A\cup B$ is at most $\left(\frac{3(\ln\ln
n)^2}{\varepsilon \ln n}\right)^{(1-o(1))c|A|/4}$ for $|A|\leq
\frac{n\ln\ln n}{\ln n}$ and at most $\delta^{(1-o(1))c|A|/4}$
otherwise. Applying the union bound we get that the probability that
there exists such $A$ (with $|N_M(A)|\leq 2|A|-1$) is at most
\begin{align*}
& \sum_{|A|< \frac{n\ln\ln n}{\ln n}}\binom{n}{|A|}\binom{n}{2|A|-1}
\left(\frac{3(\ln\ln n)^2}{\varepsilon \ln
n}\right)^{(1-o(1))c|A|/4}+\sum_{|A|=\frac{n\ln\ln n}{\ln
n}}^{\frac{10000n}{\ln\ln n}}\binom{n}{|A|}\binom{n}{2|A|-1}
\delta^{(1-o(1))c|A|/4}\\
& \leq  \sum_{|A|< \frac{n\ln\ln n}{\ln
n}}n^{3|A|}\left(\frac{3(\ln\ln n)^2}{\eps\ln
n}\right)^{\frac{24|A|\ln n}{\ln\ln n}}+\sum_{|A|=\frac{n\ln\ln
n}{\ln n}}^{\frac{10000n}{\ln\ln n}}
\left(\frac{e^3n^3}{4|A|^3}\right)^{|A|}\delta^{\frac{24|A|\ln
n}{\ln\ln n}} \\
& \leq \sum_{|A|< \frac{n\ln\ln n}{\ln
n}}\left[n^3\exp\left(\frac{24\ln n}{\ln\ln
n}\ln\left(\frac{3(\ln\ln n)^2}{\eps\ln
n}\right)\right)\right]^{|A|}+ \sum_{|A|=\frac{n\ln\ln n}{\ln
n}}^{\frac{10000n}{\ln\ln n}} \left(\alpha \frac{\ln^3 n}{\ln \ln ^3
n}\delta^{\frac{24\ln n}{\ln\ln n}}\right)^{|A|} \\
& \leq \sum_{|A|< \frac{n\ln\ln n}{\ln
n}}\left[n^3\exp\left(-(1-o(1))24\ln
n\right)\right]^{|A|}+o(1)=o(1).
\end{align*}

\textbf{Case II:} At least $c|A|/4$ edges of Maker which are
incident to $A$ were chosen by vertices from $B$. As in Case I,
notice that there are at most $o(|B|)$ vertices $v\in B$ such that
$e(v,A)\geq f(n)(\ln\ln n)^2$ when $|B|\leq \frac{2n\ln\ln n}{\ln
n}$ and at most $o(|B|)$ vertices $v\in B$ such that
$e(v,A)=\Omega(np)$ when $\frac{2n\ln\ln n}{\ln n}\leq |B|
\frac{20000n}{\ln\ln n}$. Similarly to the previous case, with the
only difference being that not all the edges which were chosen by
vertices from $B$ have to touch $A$, we get that the probability
that all Maker's edges incident to $A$ were chosen in $A\cup B$ is
at most
$$\binom{c|B|}{c|A|/4}\left(\frac{3(\ln\ln n)^2}{\varepsilon
\ln n}\right) ^{(1-o(1))c|A|/4}$$  or
$$\binom{c|B|}{c|A|/4}\delta ^{(1-o(1))c|A|/4},$$ for an arbitrarily small $\delta$, for $|B|\leq \frac{2n\ln\ln n}{\ln n}$
or $\frac{2n\ln\ln n}{\ln n}\leq |B| \frac{20000n}{\ln\ln n}$,
respectively (the binomial coefficient corresponds to the number of
possible choices of edges from $E_M(A,B)$ out of all edges chosen by
vertices from $B$). Applying the union bound, similarly to the
computations in Case I, we get that the probability that there
exists such $A$ (with $|N_M(A)|\leq 2|A|-1$) is $o(1)$.


This completes the proof that Maker can build a $(\frac{10000 n}{\ln
\ln n},2)$-expander fast and thus is able to follow Stage I of the
proposed strategy.

\textbf{Stage II:} It is enough to prove that Maker has a strategy
to ensure that $E_M(A,B)\neq \emptyset$ for every two disjoint
subsets $A,B\subseteq V$ of sizes $|A|=\frac{10000n}{\ln\ln n}$ and
$|B|=n/10$. Otherwise, there exists a subset $X\subseteq V$ of size
$\frac{10000n}{\ln\ln n}\leq |X|\leq n/5$ such that $|X\cup N(X)|<
3|X|$. In this case, there exist two subsets $A\subseteq X$ and
$B\subseteq V\setminus (X\cup N(X)) $ with $|A|=\frac{10000n}{\ln\ln
n}$ and $|B|=n/10$ such that $E_M(A,B)=\emptyset$.

Recall that by Property $(P8)$ of Theorem~\ref{PropertiesGnpDense},
$G\sim \gnp$ is typically such that for every two such subsets
$A,B\subseteq V$ and for every $\alpha>0$ we have that $$ e_G(A,B)
\geq (1-\alpha)|A||B|p\geq \frac{999n^2p }{\ln \ln n}.$$

To achieve his goal for this stage, Maker can use the trick of fake
moves and to play as Breaker in the $(\frac{np\ln \ln n}{100\ln
n},1)$ Maker-Breaker game where the winning sets are
$$\mathcal F=\{E_F(A,B): A,B\subseteq V \textrm{, }A \cap
B=\emptyset \textrm{, }|A|=\frac{10000n}{\ln\ln n} \textrm{ and
}|B|=n/10\}.$$

Notice that since so far Breaker has played at most $\frac{100n\ln
n}{\ln\ln n}$, we get that $e_F(A,B)\geq \frac{899n^2p}{\ln\ln n}$
for every $A,B\subset V(G)$ of sizes $|A|=\frac{10000n}{\ln\ln n}$
and $|B|=n/10\}$. Finally, since the following inequality holds
$$\binom{n}{\frac{10000n}{\ln\ln n}}\binom{n}{n/10} 2^{-89900n\ln
n/(\ln\ln n)^2}\leq 4^n2^{-\omega(n)}=o(1)$$ it follows by
Theorems~\ref{bwin} and~\ref{fakemoves} that indeed Maker can
achieve his goals for this stage within $\frac{e(G)}{np\ln\ln
n/100\ln n}=\frac{100n\ln n}{\ln\ln n}$ moves (recall that
$e(G)=\Theta(n^2p)$).

\textbf{Stage III:} So far Maker has played at most $\frac{200n\ln
n}{\ln\ln n}$ moves (and at least $\frac{50n\ln n}{\ln\ln n}$ moves)
and his graph is an $(n/5,2)$-expander. Notice that for the choice
$a = 2$ Lemma~\ref{l:enoughboosters} holds. In addition, Maker and
Breaker together claimed $o(n^2p)$ edges of $G$. Therefore, there
are still $\Theta(n^2p)$ free boosters in $G$, so Maker can easily
claim $n$ boosters and to turn his graph into a Hamiltonian graph.

This completes the proof that Maker wins the game $\mathcal H(G)$
(and of course also the game $\mathcal M(G)$). {\hfill $\Box$
\medskip
\\}

Now, we briefly sketch the proof of Theorem~\ref{mainSparse}.

\textbf{Sketch of proof of Theorem~\ref{mainSparse}.} Let $c > 100$,
$p=\frac{cn}{\ln n}$ and $G\sim \gnp$. The upper bound on $b^*$ is
obtained immediately from Theorem~\ref{BreakerIsolatingAvertex}. We
wish to show that $G$ is typically such that given $b\leq c/10$,
Maker has a winning strategy in the $(1,b)$ game $\mathcal H(G)$.
First, we make the following modifications to
Theorem~\ref{minimumdegree}:
\begin{itemize}
    \item In the statement of the theorem, $p = \frac{c\ln n}{n}$,
    $b \leq \frac{np}{10\ln n} = \frac{c}{10}$, and $\eps$ is some
    positive constant.
    \item By similar calculations to those in $(P1)$ of
    Theorem~\ref{PropertiesGnpDense}, we can assume that $\delta(G)\geq \frac{1}{2}np$.
    \item We conclude that $\danger_{B_g}(v_g) \geq (\frac{1}{2}-\frac{\eps}{3})np - b(2c-1)
    \geq(\frac{1}{2}-\frac{\eps}{3}-\frac{\eps}{60})np$.
    \item Finally, we use the following calculation:
        \begin{eqnarray}
        \avdanan_{B_1}(I_{g-1}) & \geq & \avdanan_{B_g}(I_{0}) + \sum_{j=1}^{r-1} \left(
        \avdanan_{B_{g-i_j}}(I_{i_j}) - \avdanan_{B_{g-i_j+1}}(I_{i_j-1})\right)\nonumber\\
        & \geq & \avdanan_{B_g}(I_{0}) - \sum_{j=1}^{r-1}\frac{2b}{j+1}\nonumber\\
        & \geq & \avdanan_{B_g}(I_{0}) - 2b\ln n\nonumber\\
        & \geq & (\frac{1}{2}-\frac{\eps}{3}-\frac{\eps}{60})np -\frac{1}{5}np\nonumber\\
        & > &0\nonumber
        \end{eqnarray}
    to get a contradiction (for sufficiently small $\eps$).
\end{itemize}

With this variant of Theorem~\ref{minimumdegree}, adapted to the
case $p=\Theta(\frac{\ln n}{n})$, the proof of Theorem
\ref{mainSparse} goes the same as the proof of
Theorem~\ref{mainDense}, mutatis mutandis.{\hfill $\Box$\medskip\\}

\textbf{Remark:} To win the $k$-connectivity game, Maker follows
Stages I and II of the proposed strategy $S_M$ with the following
parameters changes:
\begin{itemize}
    \item In Stage I, Maker creates an $(\frac{n\ln\ln n}{\ln n},
    k)$-expander.
    \item In Stage II, Maker makes his graph an $(\frac{n+k}{2k},
    k)$-expander by claiming an edge between any two disjoint
    subsets $A,B \subseteq V$ such that $|A| = \frac{n\ln\ln n}{\ln
    n}$, $|B| = \frac{n}{10k}$.
\end{itemize}

Then, by Lemma~\ref{kconnectivity}, Maker's graph is $k$-connected
and he wins the game. We omit the straightforward details and the
calculations, which are almost identical to those of the
Hamiltonicity game.

\section{Avoider-Enforcer games on $\gnp$}
\subsection{Avoider's win}

In this subsection we prove Theorem~\ref{AvoiderWin}.
%

\begin{proof} In order to isolate a vertex in his graph Avoider does the
following: before his first move (regardless of the identity of the
first player) Avoider identifies a set $U \subseteq V$ of size
$\sqrt{\frac{n}{\ln n}}$ with $e(U) \leq \frac{np}{2 \ln n}$ (he can
find such a set since this is the expected number of edges inside a
set of this size). Assume $|U| \equiv 0(\textrm{mod } 3)$ (otherwise
Avoider removes from $U$ a vertex or two and everything works the
same). Then, in his first move, Avoider claims all the edges not
incident to $U$. He then ignores -- until he can no longer do so --
all the edges inside $U$ and pretends he is Enforcer in the
following reverse box game: he divides the vertices of $U$ into
triplets -- each triplet is a box. The elements in each box are all
the edges between the three vertices of the box and $V \setminus U$.

Avoider does not claim edges inside $U$ unless he has to. Enforcer,
however, in his new role as Avoider in the reverse box game, may
claim occasionally edges inside $U$. However, since his bias is too
big, in each move he must make at least $b - e(U) \geq
\frac{24.5np}{\ln n}$ of his steps "in the boxes" (between $U$ and
its complement).

We may assume that $p \geq \frac{\ln n}{n}$, since otherwise $G\sim
\gnp$ typically contains isolated vertices and Avoider wins no
matter how he plays. Therefore, by $(P1)$ we can bound from above
the degree of every vertex in the graph by $4np$, and so the size of
each box is bounded from above by $12np$. The number of boxes in
this game is $\frac{1}{3}\sqrt{\frac{n}{\ln n}}$. Enforcer's bias is
1 and Avoider's bias is at least $\frac{24.5np}{\ln n}$. Putting it
all together in the terms of Lemma~\ref{EnforcerWinBox} we get that,
as
$$2\exp\left(\frac{12np}{\frac{24.5np}{\ln n}}\right) <
\exp(0.49\ln n) < \frac{1}{3}\sqrt{\frac{n}{\ln n}},$$ Enforcer wins
this game, i.e. Avoider is forced to claim all the elements in one
of the boxes.

Now let's go back to the original game. By what we have just shown,
as long as Avoider does not claim edges inside $U$, he has at least
three isolated vertices in his graph. So if he can avoid claiming
edges inside $U$ throughout the game, he wins. If he is forced to
claim an edge inside $U$, it means that all the remaining free edges
on the board are inside $U$. By claiming one edge he will touch at
most two of his isolated vertices, and Enforcer in his next move
will be forced to claim all the remaining edges on the board,
leaving at least one isolated vertex in Avoider's graph.
\end{proof}

\subsection{Enforcer's win}
In this subsection we prove Theorem~\ref{EnforcerWin}. For this
proof we would like to use similar techniques to those used by
Krivelevich and Szab\'o in \cite{KS}. We use the following
Hamiltonicity criterion by Hefetz et al:

\begin{lemma}[\textbf{\cite{HKS1}, Theorem 1.1}]\label{HamCrit}
Let $12 \leq d \leq e^{\sqrt[3]{\ln n}}$ and let $G$ be a graph on
$n$ vertices satisfying properties $\textbf{P1}$, $\textbf{P2}$
below:

\begin{enumerate}
\item [\textbf{P1}] For every $S \subseteq V$ , if $|S| \leq k_1(n,d) := \frac{n \ln
\ln n \ln d}{d \ln n \ln \ln \ln n}$ then $|N(S)| \geq d|S|$;
\item [\textbf{P2}] There is an edge in $G$ between any two disjoint subsets $A,B
\subseteq V$ such that $|A|,|B| \geq k_2(n,d) := \frac{n \ln \ln n
\ln d}{4130 \ln n \ln \ln \ln n}$.
\end{enumerate}

Then $G$ is Hamiltonian, for sufficiently large $n$.
\end{lemma}

Clearly, if by the end of the game Avoider's graph is Hamiltonian,
it also contains a perfect matching. In addition, the proof of
Theorem 6 in \cite{KS} shows that, in the terms of Lemma
\ref{HamCrit}, if $G$ satisfies $\textbf{P1}$ and $\textbf{P2}$ then
$G$ is $d$-connected. In particular, if $d=\omega(1)$, then $G$ is
$k$-connected for any fixed $k$. Theorem~\ref{EnforcerWin} is now an
immediate corollary of Lemma~\ref{HamCrit} and the following
theorem:

\begin{theorem}\label{ForceExpander}
Let $p \geq \frac{70000\ln n}{n}$ and $b \leq \frac{np}{20000\ln
n}$. In a biased $(1,b)$ Avoider-Enforcer game, Enforcer has a
strategy to force Avoider to create a graph satisfying \textbf{P1}
and \textbf{P2} with $d = d(n) = \ln\ln n$ provided $n$ is large
enough.
\end{theorem}

\begin{proof}
As we set $d = d(n) = \ln\ln n$ we use the following notation:
$$k_1^* = k_1^*(n) = k_1(n,d) = \frac{n}{\ln n},$$
$$k_2^* = k_2^*(n) = k_2(n,d) = \frac{n\ln\ln n}{4130\ln n}.$$

For every $1 \leq k \leq k_1^*$ and for every $S \subseteq V$, $|S|
= k$, define the hypergraph $\cF(S)$ on $N(S)$ in the following way:
divide the vertices of $N(S)$ into $2dk$ subsets, each of size
$|N(S)| / 2dk$ (by $(P5)$ and $(P4)$ the size of $N(S)$ is much
greater than $2dk$, so this is well defined). Each combination of
$dk$ subsets forms a hyperedge in $\cF(S)$.

For a given $S$ of size $k$, if by the end of the game in Avoider's
graph $S$ is connected by an edge to every hyperedge of $\cF(S)$,
then $|N_A(S)| > dk$. Otherwise, there are $dk$ subsets disconnected
from $S$ which form a hyperedge in $\cF(S)$, in contradiction. So in
order to force $\textbf{P1}$ in Avoider's graph, it suffices to
ensure that in his graph, for every $1 \leq k \leq k_1^*$, for every
$S \subseteq V$, $|S| = k$, and for every $F \in \cF(S)$, there is
an edge between $S$ and $F$.

Notice that for every $F \in \cF(S)$, $e(S,F) \geq
\frac{1}{180}|S|np$. Indeed, if $1 \leq |S| \leq \frac{80}{p}$, then
by (P4) $N(S) \geq \frac{1}{90}|S|np$, and so the number of edges in
$G$ between $S$ and half of its external neighborhood is at least
the number of vertices there, which is at least
$\frac{1}{180}|S|np$. If $\frac{80}{p} \leq |S| \leq \frac{n}{\ln
n}$, then by Lemma~\ref{AtLeast80} the number of edges in $G$
between $S$ and half of its external neighborhood is at least
$\frac{1}{50}|S|np > \frac{1}{180}|S|np$.

In order to force $\textbf{P2}$ in Avoider's graph, it is enough to
ensure that he claims an edge between any two disjoint sets of size
$k_2^*$. By $(P7)$, for any such sets $A,B$, $e(A,B) \geq
0.5(k_2^*)^2p$.

Finally, in order to conclude that Enforcer can force Avoider to
claim all these edges, by Lemma~\ref{AvoidersCrit} it is sufficient
to verify that:
$$\sum_{k=1}^{k_1^*}\sum_{|S|=k}|\cF(S)|\left(1+\frac{1}{b}\right)^
{-\frac{1}{180}|S|np} +
\sum_{|A|,|B|=k_2^*}\left(1+\frac{1}{b}\right)^{-\frac{1}{2}(k_2^*)^2p}
< \left(1+\frac{1}{b}\right)^{-b}.$$

By using the well known estimate $1+x = e^{x+\Theta(x^2)}$ for $x
\rightarrow 0$, we can bound the term on the right hand side from
below by $e^{-\frac{1}{2}}$.

The first summand on the left hand side can be estimated from above
by:
$$\sum_{k=1}^{k_1^*}\binom{n}{k}\binom{2dk}{dk} e^{-\frac{knp}{180b}} \leq
\sum_{k=1}^{k_1^*}\left[n(2e)^de^{-\frac{np}{\frac{180np}{20000\ln
n}}}\right]^k \leq $$
$$\leq \sum_{k=1}^{k_1^*}\left[ne^{2\ln\ln n}e^{-100\ln
n}\right]^k = o(1).$$ The second summand on the left hand side can
be estimated from above by:
$$\binom{n}{k_2^*}^2e^{-\frac{0.5(k_2^*)^2p}{b}} \leq
\left[\left(\frac{en}{k_2^*}\right)^2e^{-\frac{10^4k_2^*\ln
n}{n}}\right]^{k_2^*} =$$
$$=\left[\left(\frac{4130e\ln n}{\ln\ln
n}\right)^2e^{-\frac{10^4}{4130}\ln\ln n}\right]^{k_2^*} \leq
\left[\vphantom{\left(\frac{4130e\ln n}{\ln\ln n}\right)^2}\left(\ln
n\right)^{2-2.4}\right]^{k_2^*} = o(1).$$ This completes the proof.
\end{proof}

\section{Concluding remarks and open questions}
In this paper we analyzed Maker-Breaker games and Avoider-Enforcer
games played on the edge set of a random graph $G\sim \gnp$. We have
shown the following:

\textbf{Maker-Breaker games:} for $p=\omega(\frac{\ln n}{n})$, the
critical bias in the Hamiltonicity, perfect matching and
$k$-connectivity games is $b^*=\frac{\ln n}{n}$. For $p=\frac{c\ln
n}{n}$ (where $c>1$), there exist $b_1=b_1(c)$ and $b_2=b_2(c)$ such
that the critical bias in these games satisfies: $b_1 \leq b^* \leq
b_2$.

\textbf{Avoider-Enforcer games:} for $p \geq \frac{c\ln n}{n}$
(where $c>70000$), there exist $c_1$ and $c_2$ such that the
critical bias in the Hamiltonicity, perfect matching and
$k$-connectivity games satisfies: $\frac{c_1\ln n}{n} \leq b^* \leq
\frac{c_2\ln n}{n}$.

Notice that while in the first case (Maker-Breaker with
$p=\omega(\frac{\ln n}{n})$) we establish the exact threshold bias,
in the latter two (Maker-Breaker with $p=\Theta(\frac{\ln n}{n})$,
and Avoider-Enforcer) we only establish the order of magnitude of
the threshold bias. Although it is possible to achieve somewhat
better constants than those appearing in this paper, we were not
able to close the gap completely. It would be nice to get to the
exact constant in these cases as well.

\end{document}